\documentclass[11pt]{article}
\usepackage[colorlinks=true]{hyperref}

\usepackage{latexsym,amsmath,amsfonts}

\newtheorem{theorem}{Theorem}
\newtheorem{proposition}{Proposition}
\newtheorem{corollary}{Corollary}
\newtheorem{fact}{Fact}

\newtheorem{lemma}{Lemma}

\newtheorem{definition}{Definition}

\newtheorem{remark}{Remark}

\newcommand{\thmref}[1]{Theorem~\ref{thm:#1}} 
\newcommand{\lemref}[1]{Lemma~\ref{lem:#1}} 
\newcommand{\propref}[1]{Proposition~\ref{prop:#1}} 
\newcommand{\defref}[1]{Definition~\ref{def:#1}} 
\newcommand{\secref}[1]{Section~\ref{sec:#1}} 
\newcommand{\eqnref}[1]{(\ref{eq:#1})} 

\newcommand\ignore[1]{}

\def\R{\mathbb{R}} 
\def\N{\mathbb{N}} 

\newcommand{\Ex}[1]{\mathbb{E}\left[#1\right]} 
\newcommand{\Prp}[2]{\mathbb{P}_{#1}\left(#2\right)} 
\newcommand{\Exp}[2]{\mathbb{E}_{#1}\left[#2\right]} 
\newcommand{\Prpwo}[1]{\mathbb{P}_{#1}} 
\newcommand{\Expwo}[1]{\mathbb{E}_{#1}} 

\newcommand{\Ind}[1]{\mathbb{I}_{#1}} 

\renewcommand{\Pr}[1]{\mathbb{P}\left(#1\right)} 

\def\sA{\mathcal{A}}
\def\sF{\mathcal{F}}

\newcommand\QED{\ifhmode\allowbreak\else\nobreak\fi
\quad\nobreak$\Box$\medbreak}
\newcommand{\proofstart}{\par\noindent\sl Proof:\rm\enspace}
\newcommand{\proofend}{\QED\par}
\newenvironment{proof}{\proofstart}{\proofend}

\def\Tmix{{\rm t}_{\rm mix}}
\def\Thit{{\rm t}_{\rm hit}}

\def\Trel{t_{{\rm rel}}}
\def\Tmix{t_{{\rm mix}}}
\def\Thit{t_{{\rm hit}}}
\def\Tcoal{t_{{\rm coal}}}
\def\one{{\bf 1}}
\newcommand\ip[2]{\langle #1,#2\rangle}
\def\davg{d_{\rm avg}}
\def\dmax{d_{\rm max}}
\def\dmin{d_{\rm min}}

\def\taucoal{\tau_{\rm coal}}

\title{Random walks on graphs: new bounds on hitting, meeting, coalescing and returning}
\author{Roberto I. Oliveira\thanks{IMPA, Rio de Janeiro, RJ, Brazil. \texttt{rimfo@impa.br}. Supported by a {\em Bolsa de Produtividade em Pesquisa} grant from CNPq (Brazil) and a {\em Cientista do Nosso Estado} grant from FAPERJ (Brazil).}{~~}and{~}Yuval Peres\thanks{Microsoft Research, Redmond, WA, USA. \texttt{peres@microsoft.com}}}

\begin{document}
 \maketitle

\begin{abstract}We prove new results on lazy random walks on finite graphs. To start, we obtain new estimates on return probabilities $P^t(x,x)$ and the maximum expected hitting time $\Thit$, both in terms of the relaxation time. We also prove a discrete-time version of the first-named author's ``Meeting time lemma"~ that bounds the probability of random walk hitting a deterministic trajectory in terms of hitting times of static vertices. The meeting time result is then used to bound the expected full coalescence time of multiple random walks over a graph. This last theorem is a discrete-time version of a result by the first-named author, which had been previously conjectured by Aldous and Fill. Our bounds improve on recent results by Lyons and Oveis-Gharan; Kanade et al; and (in certain regimes) Cooper et al. \end{abstract}
\pagebreak

\section{Introduction}
Random walks on graphs and other reversible Markov chains are fundamental mathematical objects, with applications in algorithms and beyond. In this paper, we derive new bounds on return probabilities, hitting times, and meeting deterministic trajectories for these processes. We also bound the expected full coalescence time of multiple random walks on the same graph. 

We state our main results in terms of lazy random walks on graphs, although some of them hold for more general chains (cf. \secref{prelim}). In the paper $G=(V,E)$ is always a finite, unoriented, connected graph with vertex set $V$, with no loops or parallel edges and $n:=|V|\geq 2$ vertices. The degree of $x\in V$ is denoted by $d_x$. Minimum, maximum, and average degrees in $G$ are denoted by $\dmin$, $\dmax$ and $\davg$, respectively. The parameter $t$ always denotes discrete time and takes values in $\N=\{0,1,2,\dots\}$.

\begin{definition}\label{def:LRW}Lazy random walk (LRW) on $G$ is the Markov chain on $V$ with transition matrix $P$ given by:
\[P(x,y):=\left\{\begin{array}{ll}\frac{1}{2}, & y=x;\\ \frac{1}{2d_x}, & xy\in E;\\ 0, & \mbox{ otherwise}\end{array}\right.\;\;\;(x,y\in V).\]
We let $(X_t)_{t=0}^{+\infty}$ denote trajectories of the chain and $\Prpwo{x}$, $\Expwo{x}$, $\Prpwo{\mu}$, $\Expwo{\mu}$ denote probabilities and expectations when $X_0=x$ or $X_0$ has law $\mu$ (respectively). Finally, we denote by $\pi$ the stationary distribution:
\[\pi(x):=\frac{d_x}{\davg\,n}\,\,(x\in V).\]\end{definition}

Recall that the hitting time of $A\subset V$ is \[\tau_A:=\inf\{t\geq 0\,:\,X_t\in A\}.\] We write $\tau_x:=\tau_{\{x\}}$ for $x\in V$. The maximum expected hitting time of LRW on $G$,
\[\Thit := \max_{y,x\in V}\,\Exp{y}{\tau_x},\]
is a natural parameter for a Markov chain, with connections to electrical network theory, cover times, mixing times and many other aspects of such processes.  See \cite{LPW_book} (especially Chapter 10) for more information. 

\subsection{Hitting and returning} 

Our first two results estimate $\Thit$ and return probabilities $P^t(x,x)$ in terms of the relaxation time of $G$ (cf. \secref{prelim}). These quantities are related by the formula \cite[Proposition 10.26]{LPW_book}:
\begin{equation}\label{eq:hitformula}\pi(x)\Exp{\pi}{\tau_x} = \sum_{s=0}^{+\infty}(P^s(x,x)-\pi(x)).\end{equation}
\begin{theorem}[Proof in \secref{proof.thit}]\label{thm:thit}Under \defref{LRW},\[\Thit\leq \frac{20\,\davg}{\dmin}\,n\,\sqrt{\Trel+1}.\]\end{theorem}
\begin{theorem}[Proof in \secref{proof.thit}]\label{thm:return}Under \defref{LRW}, for any $x\in V$ and $t\geq 0$:
\begin{eqnarray*}P^t(x,x) - \pi(x)&\leq& \frac{10\,d_x}{\dmin}\,\left(\frac{1}{\sqrt{t+1}}\wedge \frac{\sqrt{\Trel+1}}{t+1}\right).\end{eqnarray*}\end{theorem}
These Theorems are specific to random walks on graphs. As we explain in Section \ref{sec:general}, the following bounds are essentially best possible for general reversible chains.

\begin{eqnarray}\label{eq:Thitgeneral}\Thit &\leq &2\max_x\left(\frac{1-\pi(x)}{\pi(x)}\right)\Trel\\ \nonumber & & \mbox{(which is of order }\frac{\davg\,n}{\dmin}\,\Trel\mbox{ for a graph);}\\
\label{eq:returngeneral}P^t(x,x) - \pi(x) &\leq& (1 - \Trel^{-1})^t\,(1-\pi(x)).\end{eqnarray} 

Kanade et al. \cite{Kanade_CoalMeet} obtained the following improvement to \eqnref{Thitgeneral} in the graph setting:
\begin{equation}\label{eq:kanade}\Thit \leq K\,\min\left\{\left(\frac{\davg}{\dmin}\right)^2 \sqrt{\Trel\log \Trel}, \left(\frac{\davg}{\dmin}\right)^{5/2}\,\Tmix\right\}\,n,\end{equation}
where $\Tmix=\Omega(\Trel)$ is the mixing time \cite[Theorem 12.5]{LPW_book} and $K>0$ is universal. \thmref{thit} has better dependence on $\Trel$ and on the ratio $\davg/\dmin$. Examples in Section \ref{sec:sharp} show that our bound is sharp in both of these parameters. 

The best previous result for return probabilities on graphs is due to Lyons and Oveis-Gharan \cite[Theorem 4.9]{LyonsOG_Spectrum}\footnote{That result is for regular graphs, but the authors explain right after the proof that it can be extended to all graphs in the form presented here.}:
\begin{equation}\label{eq:returnruss}P^t(x,x)-\pi(x)< \frac{13\,\dmax}{\dmin}\,\frac{1}{\sqrt{t+1}}.\end{equation}
Theorem \ref{thm:return} improves the constant, replaces $\dmax$ with $d_x$, and improves both \eqnref{returnruss} and \eqnref{returngeneral} when $1\ll t/\Trel\ll \log\Trel$.

\begin{remark}Anna Ben-Hamou (personal communication) improved our original constants in Theorems \ref{thm:thit} and \ref{thm:return}. We thank her for her permission to use these improvements.\end{remark}

\subsection{Results on meeting and coalescing}

Our next result bounds the probability that LRW hits a moving target in terms of $\Thit$.

\begin{theorem}[Proof in \secref{proof.meeting}]\label{thm:meeting}Under \defref{LRW}, for any $t\in\N\backslash\{0\}$ and any sequence $h_0,h_1,\dots,h_t\in V$,
\[\Prp{\pi}{\forall 0\leq s\leq t\,:\, X_s\neq h_s}\leq \left(1-\frac{1}{\Thit}\right)^t.\]\end{theorem}

This is a discrete-time version of the ``Meeting Time Lemma"~in \cite{Oliveira_TAMS}. Notice that this Theorem fails for non-lazy random walk on a bipartite graph. On the other hand, \thmref{meeting} can be extended to all reversible Markov chains with nonnegative spectrum (cf. \secref{prelim} below).  

The original application of the continuous-time version of \thmref{meeting} was to the study of coalescing random walks. Following the argument in  \cite{Oliveira_TAMS}, we may obtain a discrete-time version of the main result of that paper, which had been previously conjectured by Aldous and Fill \cite{aldousfill_book}.  

\begin{theorem}[Definitions in Section \ref{sec:tcoal}, proof sketch in Appendix \ref{sec:proof.tcoal}]\label{thm:tcoal}Under \defref{LRW}, consider $n$ coalescing random walks that evolve according to LRW on $G$. Let $\Tcoal$ denote the expected time until these walks have all coalesced into one. Then $\Tcoal\leq K\,\Thit$ with $K>0$ a universal constant.\end{theorem}

Kanade et al. \cite[Theorem 1.3]{Kanade_CoalMeet} obtained a partial form of \thmref{tcoal} where $K=O(\log (\dmax/\dmin))$ grows slowly with the ratio $\dmax/\dmin$. In particular, their result implies \thmref{tcoal} when $G$ is regular. They also obtained sharp results relating coalescing and meeting times in general \cite[Theorems 1.1 and 1.2]{Kanade_CoalMeet} that do not follow from our methods.  

\subsection{Additional comments}

We finish the introduction with additional comments on our results. Theorem \ref{thm:thit} gives a bound on $\Thit$ that may be viewed through the lens of electrical network theory. The relation between effective resistances and commute times \cite[Chapter 9]{LPW_book} gives the following immediate corollary.

\begin{corollary}[Proof omitted] Let $G$ be finite connected graph. Place unit resistances on each edge of $G$. Then the efffective resistance between any two vertices $x,y\in V$ is bounded by $K\,\sqrt{\Trel}/\dmin$, where $K>0$ is universal.\end{corollary}

Intriguingly, we do not know of a direct proof of this result via the Dirichlet form or a network reduction.

\thmref{return} on return probabilities may be used to tighten results on random walk intersections by Peres et al. \cite{PSSS_Intersections}. Let $t_I$ be the maximum expected intersection time of two independent lazy random walks over $G$ from worst-case initial states. Proposition 1.8 in  \cite{PSSS_Intersections} gives the following bound for regular graphs:
\[t_{I}\leq K\,\sqrt{n}\,\min\{t_{\rm unif},\Trel\log(\Trel+1)\}^{3/4},\mbox{ with $K>0$ universal}\]
and $t_{\rm unif}$ the {\em uniform} mixing time. With Anna Ben-Hamou, we have used \thmref{return} to obtain an improved bound that holds for all graphs:
\[t_{I}\leq K\,\sqrt{\frac{\davg\,n}{\dmin}}\,(\Trel)^{3/4},\mbox{ with $K>0$ universal,}\]
which can be shown to be sharp.  This result will appear in the full version of \cite{bop}.

Finally, the combination of Theorems \ref{thm:tcoal} and \ref{thm:thit} leads to the following corollary.

\begin{corollary}[Proof omitted] Under \defref{LRW}, let $\Tcoal$ be as in \thmref{tcoal}. Then:
\[\Tcoal \leq K\,\frac{\davg}{\dmin}\,n \,\sqrt{\Trel}, \mbox{ with $K>0$ universal}.\]\end{corollary}

This may be compared to \cite[Theorem 1]{cooper} by Cooper et al.:
\[\Tcoal \leq K\,\log^4  n\,\Trel + K\,\left(\frac{\davg^2}{\frac{1}{n}\sum_{v\in V}d_v^2}\right)\,n\,\Trel.\]
Our bound is better when $G$ is regular or (more generally) when
\[\frac{1}{n}\sum_{v\in V}d_v^2\ll \davg\,\dmin\,\sqrt{\Trel}.\]
On the other hand, Cooper et al.'s bound is stronger when degrees are not balanced (eg. in a star graph). See \cite[Section 5]{cooper} for even stronger bounds when degrees are imbalanced.

\section{Preliminaries}\label{sec:prelim}

Before we prove our results, we discuss basic properties of LRW. The books \cite{aldousfill_book,LPW_book} contain much more material on reversible chains and random walks on graphs.

LRW as in \defref{LRW} is an irreducible chain because $G$ is connected. LRW is also reversible: $\pi(x)\,P(x,y) = \pi(y)\,P(y,x)$ for all $x,y\in V$. This means that $P$, as an operator over $\R^V$, is is self-adjoint with respect to the inner product:
\[\ip{f}{g}:=\sum_{x\in V}\,\pi(x)\,f(x)\,g(x)\,\,(f,g\in\R^V).\]So $P$ has real spectrum 
\[\lambda_1=1>\lambda_2\geq \lambda_3\geq \dots \lambda_n\geq -1.\]
and fact that $P(x,x)\geq 1/2$ for all $x\in V$ implies $\lambda_n\geq 0$. That is, LRW on a finite connected graph $G$ is a {\em reversible and irreducible finite Markov chain with nonnegative spectrum}. In particular $P$ is positive semidefinite and has a positive semidefinite square root $\sqrt{P}$. The relaxation time of such a chain is $\Trel:=(1-\lambda_2)^{-1}$.

Let $\one\in\R^V$ denote the function that is equal to $1$ everywhere. To each $\lambda_i$ we may associate an eigenfunction $\Psi_i$ (with $\Psi_1=\one$) so that $\{\Psi_i\}_{i=1}^n$ is an orthonormal basis of $(\R^V,\ip{\cdot}{\cdot\cdot})$. We use $\|\cdot\|$ to denote the norm corresponding to the inner product in $\R^V$.

\subsection{General upper bounds for hitting and returning}\label{sec:general}

Using the above notation, we explain why the bounds on hitting times \eqnref{Thitgeneral} and return probabilities \eqnref{returngeneral} are nearly optimal for a reversible chain $P$ with nonnegative spectrum.

Set $\gamma:=1/\Trel$, so that $\lambda_2=1-\gamma$. Let $\Pi$ be the matrix with entries $\Pi(x,y) = \pi(y)$ ($x,y\in V$). Let $I$ denote the identity matrix. The Markov chain that, at each step, stays put with probability $1-\gamma$ and jumps to a point picked from $\pi$ with probability $\gamma$, has transition matrix:
\[P_*:= (1 -\gamma)\,I + \gamma\,\Pi\]
$P_*$ has the same eigenbasis as $P$, with eigenvalues $1$ (with multiplicity $1$, eigenvector $\one$) and $\lambda_2=1-\gamma$ (multiplicity $n-1$). Thus $P$ and $P_*$ have nonnegative spectrum, the same relaxation time, and moreover $\ip{f}{P^tf}\leq \ip{f}{P_*^tf}$ for all $t\geq 0$ and $f\in \R^V$. In particular:
\[\forall x\in V\,\forall t\geq 0\,:\,P^t(x,x)-\pi(x)\leq P_*^t(x,x) - \pi(x) = (1-\Trel^{-1})^t\,(1-\pi(x))\]
and (with $\mathbb{E}_*$ denoting expectation for $P_*$)
\[\forall x\in V\,:\,\pi(x)\Exp{\pi}{\tau_x}\leq \pi(x)\mathbb{E}_{*,\pi}[\tau_x]=(1-\pi(x))\,\Trel.\]
In particular, \eqnref{returngeneral} is exact for $P_*$, whereas \eqnref{Thitgeneral} is sharp up to a constant factor.

\subsection{Lower bounds for hitting on graphs}\label{sec:sharp}

We showed above that, in the graph setting, our Theorem \ref{thm:thit} improves upon the best-possible results for chains with nonnegative spectrum. The next two examples show that the dependence of \thmref{thit} on $\Trel$ and $\davg/\dmin$ is best possible.
\begin{enumerate}
\item {\em A streched expander.} Take a $3$-regular expander on $n_0$ vertices. Subdivide each edge into a path of length $k$. The resulting graph $G$ has $n=\Theta(kn_0)$ vertices, relaxation time $\Theta(k^2)$, $\davg/\dmin=O(1)$ and $\Thit=\Omega(n_0k^2) = \Omega(n\sqrt{\Trel})$. 
\item {\em A lolipop.} Take a $d$-regular graph $G_0$ on $n/2$ vertices which has constant relaxation time. Connect a path of length $n/2$ to one of the the vertices of $G_0$. The resulting graph $G$ has $\Trel = \Theta(n^2)$, $\davg/\dmin =\Theta(d)$ and maximum hitting time $\Thit=\Omega(dn^2)=\Omega(dn\sqrt{\Trel})$.\end{enumerate}

\section{Return probabilities and hitting times}\label{sec:proof.thit}

In this section we prove \thmref{thit} on $\Thit$ and \thmref{return} on $P^t(x,x)$. We will need two lemmas on ``Green's function":
\begin{equation}\label{eq:defgreen}g_t(x,x):=\sum_{s=0}^{t}P^s(x,x).\end{equation}
\begin{lemma}[Proof in \secref{proof.sumfint}] \label{lem:sumfint}Consider $P$, a reversible and irreducible finite Markov chain with nonnegative spectrum, stationary distribution $\pi$ and state space $V$. Fix $x\in V$ and $t\geq 0$. Then:
\begin{eqnarray*}P^t(x,x) - \pi(x)\leq  \frac{g_t(x,x) - (t+1)\pi(x)}{t+1}\leq \left(\frac{e}{e-1}\right)\,\frac{g_{(\lceil \Trel\rceil-1)\wedge t}(x,x)}{t+1}.\end{eqnarray*}\end{lemma}
\begin{lemma}[Proof in \secref{proof.green}]\label{lem:green} Under \defref{LRW}, for $x\in V$ and $0\leq t \leq \lceil \Trel\rceil$:\[\frac{g_t(x,x)}{\pi(x)}\leq 6\,\frac{\davg\,n}{\dmin}\,\sqrt{t+1}.\]\end{lemma}
Let us see how \thmref{thit} and \thmref{return} follow from the Propositions. 
\begin{proof}[of \thmref{thit}] For $x\in V$, formula \eqnref{hitformula} may be rewritten as:
\[\Exp{\pi}{\tau_x} = \frac{\lim_{t\to +\infty}(g_t(x,x) - (t+1)\pi(x))}{\pi(x)}.\] \lemref{sumfint}, \lemref{green} and the formula $\pi(x)=d_x/\davg n$ give:
\[\Exp{\pi}{\tau_x}\leq \frac{e}{e-1}\,\frac{g_{(\lceil \Trel\rceil-1)}(x,x)}{\pi(x)}\leq \left(\frac{6e}{e-1}\right)\,\frac{\davg\,n}{\dmin}\,\sqrt{\Trel+1}.\] To finish, we bound $6e/(e-1)\leq 10$ and use $\Thit\leq 2\max_{x\in V}\Exp{\pi}{\tau_x}$ \cite[Lemma 10.2]{LPW_book}.\end{proof}

\begin{proof} [of \thmref{return}] Using the same results as in the previous proof,
\[(t+1)\,(P^t(x,x) - \pi(x))\leq g_t(x,x) - (t+1)\pi(x)\leq \left(\frac{6e}{e-1}\right)\,\frac{d_x}{\dmin}\,\sqrt{(t+1)\wedge(\Trel+1)}.\]
The result follows from estimating the constant by $10$.\end{proof}

\subsection{Sum up to the relaxation time and stop}\label{sec:proof.sumfint}
\begin{proof}[of \lemref{sumfint}] We use the notation in \secref{prelim}. Note that there exists a  $f\in\R^V$ such that, for all $t\geq 0$:
\begin{equation}\label{eq:tobesummed}P^t(x,x) - \pi(x) = \ip{f}{P^tf}  - \ip{f}{\one}^2 = \sum_{i=2}^n\,\lambda_i^t\,\ip{\Psi_i}{f}^2.\end{equation}
Since $0\leq \lambda_i\leq \lambda_2<1$ for each $i\in[n]\backslash\{1\}$, $P^t(x,x)-\pi(x)$ decreases with $t$. In particular, 
\begin{equation*}0\leq P^t(x,x) - \pi(x)\leq \frac{\sum_{s=0}^{t}(P^s(x,x)-\pi(x))}{t+1}= \frac{g_t(x,x) - (t+1)\pi(x)}{t+1}.\end{equation*}
Now note that, by \eqnref{tobesummed}, \[g_{t}(x,x) - (t+1)\pi(x) = \sum_{i=2}^n\,\left(\frac{1-\lambda_i^{t+1}}{1-\lambda_i}\right)\,\ip{\Psi_i}{f}^2\leq \sum_{i=2}^n\,\left(\frac{1}{1-\lambda_i}\right)\,\ip{\Psi_i}{f}^2,\]
whereas for $t=\lceil\Trel\rceil - 1$, $0\leq \lambda_i^{t+1}\leq \lambda_2^{\Trel}\leq e^{-1}$, and
\[g_{\lceil\Trel\rceil - 1}(x,x) -\lceil\Trel\rceil\pi(x) \geq  \sum_{i=2}^n\,\left(\frac{1-e^{-1}}{1-\lambda_i}\right)\,\ip{\Psi_i}{f}^2.\]
The proof finishes by combining the two last displays.\end{proof}

\subsection{Estimate on Green's function}\label{sec:proof.green}

We prove \lemref{green} by adapting \cite[Proposition 6.16]{aldousfill_book} to non-regular graphs. We need two Propositions that we prove in Appendix \ref{sec:proof.additional} via electrical network theory.

\begin{proposition}[Proof in \secref{proof.additional}]\label{prop:maxtrel}Under \defref{LRW}, 
\[\Trel\leq \max_{x,y\in V}(\Exp{x}{\tau_y} + \Exp{y}{\tau_x})\leq 6\left(\frac{\davg}{\dmin}\right)\,n^2 -4.\]\end{proposition}

\begin{proposition}[Proof in \secref{proof.additional}]\label{prop:exitset} Under Definition \ref{def:LRW}, let $A\subset V$ be nonempty. Take $x\in V\backslash A$ and consider the number of returns to $x$ up to time $\tau_{A}-1$:
\[g_{\tau_A-1}(x,x):= \Exp{x}{\sum_{s=0}^{\tau_A-1}\,\Ind{\{X_s=x\}}}.\] Then:
\[\frac{g_{\tau_A-1}(x,x)}{\pi(x)}\leq 9\,\left(\frac{\davg\,n}{\dmin}\right)^2\,(1-\pi(A)).\]
\end{proposition}

\begin{proof}[of \lemref{green}] Fix $\alpha>1$ and define the set:
\[A_\alpha:=\{y\in V\,:\, g_t(y,x)\leq \alpha\,\pi(x)\,(t+1)\}.\]
We bound the measure of $V\backslash A_\alpha$ via Markov's inequality and reversibility: \begin{eqnarray*}1 - \pi(A_\alpha) \leq  \sum_{y\in V}\pi(y)\frac{g_t(y,x)}{\alpha\,(t+1)\pi(x)} =  \frac{\sum_{y\in V}g_t(x,y)}{\alpha\,(t+1)} = \frac{1}{\alpha}.\end{eqnarray*}
In particular, $A_\alpha\neq \emptyset$ when $\alpha>1$. We now {\em claim} that:
\begin{equation}\label{eq:greenalpha}\frac{g_t(x,x)}{\pi(x)}\leq 9\,\left(\frac{\davg\,n}{\dmin}\right)^2\,\frac{1}{\alpha} + \alpha\,(t+1).\end{equation}
In fact, assuming this we obtain the Proposition by setting 
\[\alpha = \frac{3\,\davg\,n}{\dmin\,\sqrt{t+1}},\mbox{ which is which is $>1$ when $t\leq \lceil \Trel\rceil$ (cf. \propref{maxtrel}).}\]

To prove \eqnref{greenalpha}, we may assume $x\not\in A_\alpha$. Note that the number of returns to $x$ up to time $t$ is at most the sum of returns up to time $\tau_{A_\alpha}-1$ with those occurring at times $\tau_{A_\alpha}\leq s\leq \tau_{A_\alpha}+t$. The strong Markov property gives:
\begin{eqnarray*}\frac{g_t(x,x)}{\pi(x)}\leq \frac{g_{\tau_{A_\alpha}-1}(x,x)}{\pi(x)} +  \Exp{x}{\frac{g_t(X_{\tau_{A_\alpha}},x)}{\pi(x)}},\end{eqnarray*}
and inequality \eqnref{greenalpha} follows from  \propref{exitset} (applied to the first term in the RHS) and the definition of $A_\alpha$ (applied to the second term in the RHS).  \end{proof}

\section{The Meeting Time Theorem in general form}\label{sec:proof.meeting}

We now present a more general version of the Meeting Time Theorem (\thmref{meeting} above).  We use the same notation and definitions for trajectories, hitting times \&c that we defined for LRW.  We also take the material and notation from \secref{prelim} for granted.

\begin{theorem}\label{thm:meetgeneral}Consider a reversible and irreducible finite Markov chain with nonnegative spectrum $P$ which is defined over a finite set $V\neq \emptyset$ and has stationary measure $\pi$. Then for any $t\geq 0$ and any choice of $h_0,\dots,h_t\in V$:
\[\Prp{\pi}{\forall 0\leq s\leq t\,:\, X_s\neq h_s}\leq \left(1 - \frac{1}{\Thit}\right)^t.\]\end{theorem}
\begin{proof}Given a linear operator $A:\R^V\to\R^V$, we denote its operator norm by:
\[\|A\|_{\rm op}:=\sup\{\|Af\|\,:\,f\in\R^V,\,\|f\|=1\} = \sup\{\ip{g}{Af}\,:\,f,g\in\R^V,\,\|f\|=\|g\|=1\} .\]
For $h\in V$, let $D_h$ be a $|V|\times |V|$ diagonal matrix that that has $1$'s at entries $(x,x)$ with $x\neq h$ and a zero at the entry $(h,h)$. $D_h$ is self-adjoint with respect to $\ip{\cdot}{\cdot\cdot}$. Defining
\[M_t := D_{h_0}\,P\,D_{h_1}\,P\,D_{h_2}\,\dots D_{h_{t-1}}\,P\,D_{h_t},\]
we see that
\[\Prp{\pi}{\forall 0\leq s\leq t\,:\, X_s\neq h_s} = \ip{\one}{M_t\one}\leq \|M_t\|_{\rm op}.\]
Thus it suffices to estimate the norm of $M_t$. We do this in two steps.
\begin{itemize}
\item[{\bf Step 1:}] $\|M_t\|_{\rm op}\leq (\max_{h\in V}\|D_h\,P\,D_h\|_{\rm op})^t$;
\item[{\bf Step 2:}] For any $h\in V$, $\|D_h\,P\,D_h\|_{\rm op}\leq 1 - \Thit^{-1}$.
\end{itemize}
For step 1 we employ $\sqrt{P}$ (defined in \secref{prelim}) and the fact that $D_{h_s}=D_{h_s}^2$ for each $s$. The norm of $M_t$ may be bounded by:
\begin{eqnarray}\nonumber \|M_t\|_{\rm op} = \left\|\prod_{s=0}^{t-1}\,(D_{h_s}\,P\,D_{h_{s+1}})\right\|_{\rm op} &=&  \left\|\prod_{s=0}^{t-1}\,(D_{h_s}\sqrt{P}\,\sqrt{P}\,D_{h_{s+1}})\right\|_{\rm op}\\ \label{eq:submult} \mbox{($\|\cdot\|_{\rm op}$ submultiplicative)}&\leq & \prod_{s=0}^{t-1}\,(\|D_{h_s}\sqrt{P}\|_{\rm op}\,\|\sqrt{P}\,D_{h_{s+1}}\|_{\rm op}).\end{eqnarray}
This last product of norms contains $2t$ terms of the form
$\|D_{h}\sqrt{P}\|_{\rm op}$  or $\|\sqrt{P}D_{h}\|_{\rm op}$ for elements $h\in V$. Now, for {\em any} linear operator $A$ with adjoint $A^\dag$, $\|A^\dag\|_{\rm op}$ and $\|A\|_{\rm op}$ are both equal to $\sqrt{\|A^\dag A\|_{\rm op}}$. Apply this to $A:=\sqrt{P}\,D_h$ and obtain:
\[\|D_{h}\sqrt{P}\|_{\rm op} = \|\sqrt{P}D_{h}\|_{\rm op} = \sqrt{\|D_{h}\,P\,D_h\|_{\rm op}}.\]
Thus step 1 follows from \eqnref{submult}. For step 2,  we fix $h\in V$. Notice that $D_h\,P\,D_h$ is self-adjoint and positive semidefinite (because $P$ is), so $\|D_h\,P\,D_h\|_{\rm op}$ equals the largest eigenvalue of $D_h\,P\,D_h$. According to \cite[Section 3.6.5 and Theorem 3.33]{aldousfill_book}, the largest eigenvalue is equal to $1 - \Exp{q_h}{\tau_h}^{-1}$ for some quasistationary disitribution $q_h$ over $V\backslash\{h\}$. Since $\Exp{q_h}{\tau_h}\leq \Thit$, the result follows. \end{proof}
 
\section{Coalescing random walks}\label{sec:tcoal}

In this section we present a generalization of \thmref{tcoal}.
\begin{theorem}[Proof sketch in Appendix \ref{sec:proof.tcoal}]\label{thm:coal2} Consider a reversible and irreducible finite Markov chain with nonnegative spectrum $P$ which is defined over a finite set $V\neq \emptyset$. Let $\taucoal$ denote the full coalescence time of a system of coalescing random walks that evolve according to $P$. Let $\Tcoal:=\Ex{\taucoal}$. Then:
\[\Tcoal\leq K\,\Thit, \mbox{ where $K>0$ is universal.}\]\end{theorem} 

\ignore{As stated above, this is the discrete-time version of the main result of \cite{Oliveira_TAMS}, which proved a conjecture of Aldous and Fill \cite{aldousfill_book}. The present proof is a simple adaptation of the argument in \cite{Oliveira_TAMS} where our \thmref{meetgeneral} replaces the ``meeting time lemma" in that paper. }

Let us define the process, postponing its analysis to Appendix \ref{sec:proof.tcoal}. Following \cite[Section 3.3]{Oliveira_TAMS}, we use a definition in terms of ``killed particles". That is, when several particles meet on the same vertex at the same time, only one the particle with lower index survives (one can alternatively think that the two particles coalesce). To make this formal, consider a Markov chain $P$ on a finite set $V$. Write $V=\{v_1,\dots,v_n\}$ and let \[(X_{t}(a))_{t\geq 0}\,:\,a=1,2,3,\dots,n\]
be $n$ independent trajectories on $V$ evolving according to $P$, each with initial state $X_0(i)=v_i$. Now define killed trajectories $Y_t(a)$ as follows. Let $\partial\not\in V$ be a ``coffin state". We set $Y_t(1)=X_t(1)$ for all $t$. Given $1<a\leq n$, assume $Y_s(1),\dots,Y_s(a-1)$ have been defined. Now let\footnote{This time $\kappa_a$ is called $\tau_a$ in \cite{Oliveira_TAMS}.}:
\[\kappa_a:=\inf\{t\geq 0\,:\, X_t(a) = Y_t(b)\mbox{ for some }b<a\}.\]
For $t\geq 0$, we set
\[Y_t(a):=\left\{\begin{array}{ll} X_t(a), & t<\kappa_a; \\ \partial, & t\geq \kappa_a.\end{array}\right.\]

One can check that the set valued process 
\[S_t:=\{Y_t(a)\,:\,1\leq a\leq n, Y_t(a)\neq \partial\}\]
is a time-homogeneous Markov chain on $2^V\backslash \{\emptyset\}$. It is this process that we call ``coalescing random walks"~evolving according to $P$". \ignore{We note that $(S_t)_{t\geq 0}$ is adapted to the filtration: 
\[\sF_t:=\sigma(X_s(i)\,:\, i\in [n],\, 0\leq s\leq t)\;\;(t\geq 0)\]
and that the $\kappa_a$ are all stopping times. } The {\em full coalescence time}\footnote{Called $C$ in \cite{Oliveira_TAMS}.} is:
\[\taucoal:= \inf\{t\geq 0\,:\, |S_t|=1\}.\]

\appendix

\section{Appendix}
\subsection{Auxiliary results via electrical network theory}\label{sec:proof.additional}

We prove here Propositions \ref{prop:maxtrel} and \ref{prop:exitset}. As noted in the main text, these results adapt the reasoning of \cite[Proposition 6.16]{aldousfill_book} to non-regular graphs. 

We work under \defref{LRW}. We will apply the theory of electrical networks \cite[Chapter 9]{LPW_book} by assigning conductance $c(a,b)=\pi(a)\,P(a,b)$ to each pair $(a,b)\in V^2$. In particular, $c(a)=\pi(a)$ and $c(a,b)=1/2\davg n$ for each pair $a\neq b$ with $ab\in E$. We will also use a standard {\em path fact} that follows from the argument in the end of the proof of \cite[Proposition 10.16, part (b)]{LPW_book}.

\begin{fact}[Path fact]\label{fact:path} If $x_0,\dots,x_\ell$ is a geodesic path in $G$, then $\ell\leq 3n/\dmin-1$.\end{fact}

\begin{proof}[of \propref{maxtrel}] The first inequality follows from \cite[Lemma 12.17]{LPW_book}. To bound the maximum commute time, we note that, by a simple network reduction argument, \[\max_{x,y\in V}(\Exp{x}{\tau_y} + \Exp{y}{\tau_x})\] is at most the diameter of $G$ times the resistance of a single edge. The latter quantity is $2\davg n$, whereas the diameter is $\leq 3n/\dmin-1$ by the Path Fact. Since $\davg\,n=2|E|\geq 2$, the Proposition follows. \end{proof}

\begin{proof}[of \propref{exitset}] Thanks to \cite[Lemma 9.6]{LPW_book} and a standard network reduction:
\[\frac{g_{\tau_A-1}(x,x)}{\pi(x)}=R_{\rm eff}(x\leftrightarrow A)\]
where $R_{\rm eff}$ denotes effective resistance. Letting $B:=V\backslash A$, our main goal will be to show:
\[\mbox{\bf Goal: }\;R_{\rm eff}(x\leftrightarrow A)\leq \frac{9\davg n}{\dmin}\,|B|,\]
as the result then follows from the fact that:
\[|B|\leq \sum_{x\in B}\frac{d_x}{\dmin} = \frac{\davg\,n}{\dmin}\,\pi(B).\]

To bound the effective resistance, consider first the case $|B|\leq 2\dmin/3$. In that case, $x$ has at least $d_x/3$ neighbors in $A=V\backslash B$, and (using $|B|\geq 1$)
\[R_{\rm eff}(x\leftrightarrow A)\leq \frac{1}{\sum_{a\in A}c(x,a)} \leq \frac{2\davg n}{\frac{d_x}{3}}= 6\,\frac{\davg n}{d_x}< \frac{9\davg n}{\dmin}\,|B|.\]

We now consider the case $|B|>2\dmin/3$. Let $y\in A$ be as close as possible to $x$ (in the graph distance). Since each edge $ab$ has resistance $2\davg\,n$, we have:  
\begin{equation}\label{eq:Reff1}R_{\rm eff}(x\leftrightarrow A)\leq R_{\rm eff}(x\leftrightarrow y)\leq (2\davg\,n)\,{\rm dist}(x,y).\end{equation}
We now bound the graph distance $k:={\rm dist}(x,y)$. Let $B:=V\backslash A$. If $x=x_0,\dots,x_k=y$ is a shortest path from $x$ to $y$, then each of the vertices $x_0,x_1,\dots,x_{k-2}\in B$ has all of their neighbors in $B$, because all points of $A$ are at distance $\geq k$ from $x$. We apply the Path Fact to the geodesic path $x_0,x_1,\dots,x_{k-2}$ in the induced subgraph $G[B]$ and obtain: \[k \leq \frac{3|B|}{\dmin}+1< \left(3 + \frac{3}{2}\right)\,\frac{|B|}{\dmin},\]
because we are assuming $|B|>2\dmin/3$. We deduce from \eqnref{Reff1} that:
\begin{equation*}\label{eq:Reff}R_{\rm eff}(x\leftrightarrow A)\leq \frac{9\,\davg\,n}{\dmin}\,|B|,\end{equation*}
as desired.\end{proof}

\subsection{Proof sketch for coalescing random walks}\label{sec:proof.tcoal}

We sketch here the proof of Theorem \ref{thm:coal2} on coalescing random walks. This Theorem generalizes Theorem \ref{thm:tcoal} from the Introduction.

\begin{proof} [Sketch of \thmref{coal2}] We explain how the proof of \cite{Oliveira_TAMS} can be modified to work in the discrete time setting. The notation and definitions from Section \ref{sec:tcoal} are taken for granted. 

As in \cite[Proposition 4.1]{Oliveira_TAMS}, it suffices to prove that: 
\[\Pr{\taucoal\geq c\,(\Tmix+\Thit)}\leq 1-\gamma\]
for some universal $c,\gamma>0$. We do this by considering a process where less particles die. That is, assume that for each $t\geq 0$ we have a set $\sA_t\subset [n]^2$ of `allowed killings". Define new killed process $Y^{\sA}_\cdot(i)$, $i\in[n]$, by setting $Y^{\sA}_t(1)=X_t(1)$ for all $i$; redefining $\kappa_a$ as
\[\kappa^{\sA}_s:=\inf\{t\geq 0\,:\, X_t(a) = Y_t(b)\mbox{ for some }b<a\mbox{ with }(b,a)\in \sA_t\};\]
and defining $Y^{\sA}_t(a)$ accordingly. As in \cite[Proposition 3.4]{Oliveira_TAMS}, we may observe that
the full coalescence time $\taucoal^{\sA}$ of this modified process dominates $\taucoal$ in the sense that $\Pr{\taucoal\geq t}\leq \Pr{\taucoal^{\sA}\geq t}$ for all $t\geq 0$.

We then apply the same strategy as in \cite[Section 4.2]{Oliveira_TAMS}. We partition $[n] = A_0\cup A_1\cup \dots \cup A_m$ so that for each $i\in[m]$ the elements of $A_{i-1}$ all preceeded all those of $A_i$ (so $A_0=\{1\}$) and $|A_i|=2^i$ for $i\leq m-1$ (so $m$ is of order $\log n$). We define a set of epochs by setting $t_\infty=0$ and defining $t_m$, $t_{m-1}$, $t_{m-2}$, $\dots$, $t_0$ as:
\begin{eqnarray*}t_m&:=&1+ \lceil 2\Tmix\rceil,\\ t_j &=&  1+ t_{j-1} + \left\lceil\frac{2^4\ln 5}{2^j}\,\Thit\right\rceil.\end{eqnarray*} 

This follows the definition in the paper except for ``$1 + \lceil\cdot\rceil$" terms. The sets of allowed killings $\sA_t$ are defined in the same way as in that paper:
\begin{enumerate}
\item {\em Epoch \# $\infty$:} $\sA_t=\emptyset$ for $t<t_m$;
\item {\em Epochs \# $m$ through \# $1$}: $\sA_t \equiv A_{j-1}\times \cap_{i=j}^mA_i$ for $t_{j}\leq t<t_{j-1}$.
\item {\em Epoch \# $0$:} set $\sA_{t}=[n]^2$
\end{enumerate}
Letting $c_0,c$ denote universal constants, we have 
\[t_0:=\sum_{i=1}^mt_i \leq 2\Tmix + 2(m+1) + c_0\,\Thit\leq c\,(\Tmix + \Thit )\]
where we have $m=O(\log n)$ whereas $\Tmix$ is at least linear in $n$. The proofs of \cite[Propositions 4.2 - 4.4]{Oliveira_TAMS} and go through as in that paper when one uses our version of the Meeting Time Lemma, \thmref{meetgeneral}. One can then finish the proof as in that article.\end{proof}

\textit{•}

\end{document}